\documentclass{amsart}%
\usepackage{amsfonts}
\usepackage{amsmath}
\usepackage{eucal}
\usepackage{amssymb}
\usepackage{graphicx}%
\providecommand{\U}[1]{\protect\rule{.1in}{.1in}}
\theoremstyle{definition}
\newtheorem{xx}{xx}[section]

\theoremstyle{plain}
\newtheorem{corollary}[xx]{Corollary}
\theoremstyle{remark}

\theoremstyle{definition}
\newtheorem{definition}[xx]{Definition}
\newtheorem{example}[xx]{Example}
\theoremstyle{remark}

\theoremstyle{plain}
\newtheorem{lemma}[xx]{Lemma}
\theoremstyle{remark}

\theoremstyle{plain}
\newtheorem{proposition}[xx]{Proposition}
\theoremstyle{remark}
\newtheorem{remark}[xx]{Remark}

\theoremstyle{plain}
\newtheorem{theorem}[xx]{Theorem}
\numberwithin{equation}{xx}
\numberwithin{equation}{section}
\DeclareMathOperator{\der}{Der}
\DeclareMathOperator{\enm}{End}
\DeclareMathOperator{\hmm}{Hom}

\DeclareMathOperator{\ic}{ic}

\begin{document}
\title[Lie-Rinehart cohomology]{Lie-Rinehart cohomology and integrable connections on modules of rank one}
\author{Eivind Eriksen}
\address{Oslo University College}
\email{eeriksen@hio.no}
\author{Trond St\o len Gustavsen}
\address{Norwegian School of Management}
\email{trond.s.gustavsen@bi.no}
\date{\today}

\begin{abstract}
Let $k$ be an algebraically closed field of characteristic $0$, let $R$ be a
commutative $k$-algebra, and let $M$ be a torsion free $R$-module of rank one
with a connection $\nabla$. We consider the Lie-Rinehart cohomology with
values in $\enm_{R}(M)$ with its induced connection, and give an
interpretation of this cohomology in terms of the integrable connections on
$M$. When $R$ is an isolated singularity of dimension $d\geq2$, we relate the
Lie-Rinehart cohomology to the topological cohomology of the link of the
singularity, and when $R$ is a quasi-homogenous hypersurface of dimension two,
we give a complete computation of the cohomology.

\end{abstract}
\maketitle

\section{Introduction}

Rinehart introduced Lie-Rinehart cohomology in Rinehart \cite{Rin:63} as a
generalization of de Rham cohomology. Later, Lie-Rinehart cohomology has been
considered by several authors, see for instance Huebschmann \cite{Hue:99} and
Casas, Ladra and Pirashvili \cite{CasLadPir:05}. In singularity theory, Huang,
Luk and Yau \cite{HuaLukYau:03} studied the so-called punctured de Rham
cohomology, and although it is not mentioned in their paper, it turns out that
this cohomology coincides with the Lie-Rinehart cohomology.

The purpose of this paper is to study the Lie Rinehart cohomology when $R$ is
a representative of an isolated singularity, and to interpret this cohomology
in terms of integrable connections on $R$-modules of rank one. The emphasis is
on explicit results and examples.

Let $k$ be an algebraically closed field of characteristic $0$, let $R$ be a
commutative $k$-algebra and let $M$ be a torsion free $R$-module of rank one
with a (not necessarily integrable) connection $\nabla:\der_{k}(R)\rightarrow
\enm_{k}(M).$ We consider the Lie-Rinehart cohomology $\operatorname{H}%
_{\operatorname{Rin}}^{\ast}(\der_{k}(R),\enm_{R}(M))$ where $\enm_{R}(M)$ has
the (integrable) connection induced by $\nabla.$ We give the following
interpretation of the Lie-Rinehart cohomology in terms of integrable connections:

\vspace{0.2cm}

\noindent\textbf{Theorem A. }\textit{Let }$R$\textit{ be a complete reduced
local }$k$\textit{-algebra and let }$M$\textit{ be a torsion free }%
$R$\textit{-module of rank one. Then we have:}

\begin{enumerate}
\item \textit{There is a canonical obstruction class }$\ic(M)\in
\operatorname{H}_{\operatorname{Rin}}^{2}(\der_{k}(R),\enm_{R}(M))$\textit{
which vanishes if and only if there is an integrable connection on }$M.$

\item \textit{If }$\ic(M)$\textit{ vanishes, then }$\operatorname{H}%
_{\operatorname{Rin}}^{1}(\der_{k}(R),\enm_{R}(M))$\textit{ is the moduli
space of integrable connections on }$M$\textit{ up to equivalence.}
\end{enumerate}

\vspace{0.2cm}

\noindent Using a spectral sequence, we relate $\operatorname{H}%
_{\operatorname{Rin}}^{\ast}(\der_{k}(R),\enm_{R}(M))$ to the topological
cohomology of the link of the singularity:\newline

\vspace{0.2cm}

\noindent\textbf{Theorem B. }\textit{Let }$R$\textit{ be a finitely generated
Cohen-Macaulay domain over }$\mathbb{C}$\textit{ of dimension }$d\geq
2$\textit{ with a unique isolated singularity }$x\in X=\operatorname{Spec}%
(R).$\textit{ Then }$\enm_{R}(M)\cong R$ and \textit{there is a natural exact
sequence}%
\[
0\rightarrow\operatorname{H}_{\operatorname{Rin}}^{1}(\operatorname{Der}%
(R),R)\rightarrow\operatorname{H}^{1}(U_{\operatorname{an}},\mathbb{C}%
)\rightarrow E_{2}^{0,1}\rightarrow\operatorname{H}_{\operatorname{Rin}}%
^{2}(\operatorname{Der}(R),R)\rightarrow\operatorname{H}^{2}%
(U_{\operatorname{an}},\mathbb{C})
\]
\textit{where }$E_{2}^{0,1}=\ker(\operatorname{H}^{1}(U,\mathcal{O}%
_{X})\rightarrow\operatorname{H}^{1}(U,\Omega_{X}^{1}))$\textit{ and
}$U=X\setminus\{x\}$\textit{. Moreover, if }$d\geq3,$\textit{ then
}$\operatorname{H}_{\operatorname{Rin}}^{1}(\operatorname{Der}(R),R)\cong%
\operatorname{H}^{1}(U_{\operatorname{an}},\mathbb{C})$\textit{ and
}$\operatorname{H}_{\operatorname{Rin}}^{2}(\operatorname{Der}%
(R),R)\hookrightarrow\operatorname{H}^{2}(U_{\operatorname{an}},\mathbb{C}).$

\vspace{0.2cm} We give an example to show that $\operatorname{H}%
_{\operatorname{Rin}}^{1}(\operatorname{Der}(R),R)$ and $\operatorname{H}%
^{1}(U_{\operatorname{an}},\mathbb{C})$ are not in general isomorphic in
dimension $d=2$. \ To further clarify the case $d=2,$ we show the following result:

\vspace{0.2cm}

\noindent\textbf{Theorem C.} \textit{Let }$R=k[x_{1},x_{2},x_{3}]/(f)$\textit{
be an integral quasi-homogenous surface singularity. The grading on }%
$R$\textit{ induces a grading on }$\operatorname{H}_{\operatorname{Rin}}%
^{\ast}(\operatorname{Der}_{k}(R),R),$\textit{ and }%
\begin{align*}
\operatorname{H}_{\operatorname{Rin}}^{0}(\operatorname{Der}_{k}(R),R)  &
=\operatorname{H}_{\operatorname{Rin}}^{0}(\operatorname{Der}_{k}%
(R),R)_{0}\cong R_{0}=k\\
\operatorname{H}_{\operatorname{Rin}}^{i}(\operatorname{Der}_{k}(R),R)  &
=\operatorname{H}_{\operatorname{Rin}}^{i}(\operatorname{Der}_{k}%
(R),R)_{0}\cong R_{d-d_{1}-d_{2}-d_{3}}\text{ \textit{for} }i=1,2\\
\operatorname{H}_{\operatorname{Rin}}^{i}(\operatorname{Der}_{k}(R),R)  &
=0\text{ \textit{for} }i\geq3,
\end{align*}
\textit{as graded }$k$\textit{-vector space, where }$d_{i}=\deg x_{i}$\textit{
for }$i=1,2,3$\textit{.}

\vspace{0.2cm}

\noindent In particular, we have that $\operatorname{H}_{\operatorname{Rin}%
}^{i}(\operatorname{Der}_{k}(R),R)\cong R_{d-1-1-1}=R_{d-3}$\ when $R$ is the
cone over a plane curve, so that%
\[
\dim_{\mathbb{C}}\operatorname{H}_{\operatorname{Rin}}^{i}(\operatorname{Der}%
_{k}(R),R)=\frac{(d-1)(d-2)}{2}%
\]
is the genus of the curve $V(f)$ in $\mathbb{P}^{2}$ for $i=1,2.$

In Huang, Luk and Yau \cite{HuaLukYau:03}, the punctured local de Rham
cohomology $H_{h}^{i}(V,x)$ of a germ of a complex analytic space was studied.
When $(V,x)$ is a hypersurface singularity of dimension $d\geq2$ with local
ring $R,$ $H_{h}^{i}(V,x)$ coincides with $\operatorname{H}%
_{\operatorname{Rin}}^{i}(\operatorname{Der}_{\mathbb{C}}(R),R)$ for $i\geq1,$
and therefore
\begin{align*}
&  \dim_{\mathbb{C}}\operatorname{H}_{\operatorname{Rin}}^{i}%
(\operatorname{Der}_{\mathbb{C}}(R),R)=0\text{ for }1\leq i\leq d-2,\text{
and}\\
&  \dim_{\mathbb{C}}\operatorname{H}_{\operatorname{Rin}}^{d}%
(\operatorname{Der}_{\mathbb{C}}(R),R)-\dim_{\mathbb{C}}\operatorname{H}%
_{\operatorname{Rin}}^{d-1}(\operatorname{Der}_{\mathbb{C}}(R),R)=\mu-\tau
\end{align*}
where $\mu$ is the Milnor number and $\tau$ is the Tjurina number of the
singularity. In particular, $\dim_{\mathbb{C}}\operatorname{H}%
_{\operatorname{Rin}}^{2}(\operatorname{Der}_{\mathbb{C}}(R),R)=\dim
_{\mathbb{C}}\operatorname{H}_{\operatorname{Rin}}^{1}(\operatorname{Der}%
_{\mathbb{C}}(R),R)$, when $R$ is a quasi-homogenous surface, and this is in
agreement with our Theorem C.

We also prove that if $R$ is a curve, then any connection on a torsion free
$R$-module (of any rank) is integrable. Moreover, if $R=k[\Gamma]$ is an
affine monomial curve and $M$ is a graded torsion free $R$-module of rank one
with a connection $\nabla$, then $\operatorname{H}_{\operatorname{Rin}}%
^{i}(\operatorname{Der}_{k}(R),M,\nabla)=0$ for $i\geq1$.

\section{Basic definitions}

Let $k$ be an algebraically closed field of characteristic $0$, and let $R$ be
a commutative $k$-algebra. A \emph{Lie-Rinehart algebra} of $R/k$ is a pair
$(\mathsf{g},\tau)$, where $\mathsf{g}$ is an $R$-module and a $k$-Lie
algebra, and $\tau:\mathsf{g}\rightarrow\der_{k}(R)$ is a morphism of
$R$-modules and $k$-Lie algebras, such that
\[
\lbrack D,rD^{\prime}]=r[D,D^{\prime}]+\tau_{D}(r)\;D^{\prime}%
\]
for all $D,D^{\prime}\in\mathsf{g}$ and all $r\in R$, see Rinehart
\cite{Rin:63}. A Lie-Rinehart algebra is the algebraic analogue of a \emph{Lie
algebroid}, and it is also known as a Lie pseudo-algebra or a Lie-Cartan pair.

When $\mathsf{g}$ is a subset of $\der_{k}(R)$ and $\tau:\mathsf{g}%
\rightarrow\der_{k}(R)$ is the inclusion map, the pair $(\mathsf{g},\tau)$ is
a Lie-Rinehart algebra if and only if $\mathsf{g}$ is closed under the
$R$-module and $k$-Lie algebra structures of $\der_{k}(R)$. We are mainly
interested in Lie-Rinehart algebras of this type, and we will usually omit
$\tau$ from the notation.

Let $\mathsf{g}$ be a Lie-Rinehart algebra. For any $R$-module $M$, we define
a \emph{$\mathsf{g}$-connection} on $M$ to be an $R$-linear map $\nabla
:\mathsf{g}\rightarrow\enm_{k}(M)$ such that
\begin{equation}
\nabla_{D}(rm)=r\nabla_{D}(m)+D(r)\;m \label{e:dp}%
\end{equation}
for all $D\in\mathsf{g},\;r\in R,\;m\in M$. A \emph{connection} on $M$ is a
$\mathsf{g}$-connection on $M$ with $\mathsf{g}=\der_{k}(R)$.

If $\nabla$ is a $\mathsf{g}$-connection on $M$, then any $\mathsf{g}%
$-connection on $M$ is given by $\nabla+ P$, where $P \in\hmm_{R}%
(\mathsf{g},\enm_{R}(M))$. The linear maps in $\hmm_{R}(\mathsf{g}%
,\enm_{R}(M))$ are called \emph{potentials}.

Let $\nabla$ be a $\mathsf{g}$-connection on $M$. We define the
\emph{curvature} of $\nabla$ to be the $R$-linear map $K_{\nabla}%
:\mathsf{g}\wedge\mathsf{g}\rightarrow\enm_{R}(M)$ given by
\[
K_{\nabla}(D\wedge D^{\prime})=[\nabla_{D},\nabla_{D^{\prime}}]-\nabla
_{\lbrack D,D^{\prime}]}%
\]
for all $D,D^{\prime}\in\mathsf{g}$. We say that $\nabla$ is an
\emph{integrable $\mathsf{g}$-connection} if $K_{\nabla}=0$.

Let $\mathsf{g}$ be a Lie-Rinehart algebra of $R/k$, and let $(M,\nabla)$ be
an $R$-module with an integrable $\mathsf{g}$-connection. Define
\[
\operatorname{C}\nolimits_{\operatorname{Rin}}^{n}(\mathsf{g},M,\nabla
)=\operatorname*{Hom}\nolimits_{R}(\wedge_{R}^{n}\mathsf{g},M)\text{ \ for
}n\geq0,
\]
and define%
\[
d^{n}:\operatorname*{C}\nolimits_{\operatorname*{Rin}}^{n}(\mathsf{g}%
,M,\nabla)\rightarrow\operatorname*{C}\nolimits_{\operatorname*{Rin}}%
^{n+1}(\mathsf{g},M,\nabla)
\]
by%
\begin{multline*}
d^{n}(\xi)(D_{0}{\wedge}\cdots{\wedge}D_{n})={\sum_{i=0}^{n}(-1)^{i}%
\nabla_{D_{i}}(\xi(D_{0}{\wedge}\cdots{\wedge}\widehat{D}_{i}{\wedge}%
\cdots{\wedge}D_{n}))+}\\
{\sum_{0\leq j<k\leq n}(-1)^{j+k}\xi([D_{j},D_{k}]{\wedge}D_{0}{\wedge}%
\cdots{\wedge}\widehat{D}_{j}{\wedge}\cdots{\wedge}\widehat{D}_{k}{\wedge
}\cdots{\wedge}D_{n})}%
\end{multline*}
Since $K_{\nabla}=0$, one verifies that $\operatorname*{C}%
\nolimits_{\operatorname*{Rin}}^{\ast}(\mathsf{g},M,\nabla)$ is a complex of
$k$-vector spaces; notice that the maps $d^{n}$ are not $R$-linear. The
\emph{Lie-Rinehart cohomology} $\operatorname{H}_{\operatorname{Rin}}^{\ast
}(\mathsf{g},M,\nabla)$ is defined to be the cohomology of the complex
$\operatorname{C}\nolimits_{\operatorname{Rin}}^{\ast}(\mathsf{g},M,\nabla),$
see Rinehart \cite{Rin:63}.

The Lie-Rinehart cohomology is also referred to as the Rinehart cohomology and
as Chevalley-Hochschild cohomology. Note that the definition of the
Lie-Rinehart cohomology generalizes the definition of Chevalley-Eilenberg
cohomology (when $R=k$) and deRham cohomology (when $R$ is regular and
$\mathsf{g}=\der_{k}(R)$). For further properties of Lie-Rinehart cohomology,
see Huebschmann \cite{Hue:99} and Casas, Ladra and Pirashvili
\cite{CasLadPir:05}.

\section{Connections and cohomology}

Assume that $R$ is a reduced noetherian $k$-algebra, and that $M$ is a rank
one torsion free finitely generated $R$-module. In this section, we relate the
set of integrable connections on $M$ to Lie-Rinehart cohomology.

Let $S\subset R$ be the set of regular elements, and let $Q=Q(R)=S^{-1}R$ be
the total ring of fractions. Then by assumption, $M\rightarrow M\otimes_{R}Q$
is injective and $M\otimes_{R}Q$ is a free $Q$-module of rank one. Moreover,
fixing an embedding $M\subseteq Q,$ we identify $\overline{R}:=\enm_{R}(M)$ as
$\{\psi\in Q:\psi M\subseteq M\}.$ We view the $R$-algebra $\overline{R}$ as a
\emph{commutative} extension of $R$ with $R\subseteq\overline{R}\subseteq Q.$

Let $\mathsf{g}$ be a Lie-Rinehart algebra and assume that there is a
$\mathsf{g}$-connection $\nabla$ on $M$. Then there is an induced $\mathsf{g}%
$-connection $\overline{\nabla}$ on $\overline{R}=\enm_{R}(M)$, given by
$\overline{\nabla}_{D}(\phi)=\nabla_{D}\,\phi-\phi\nabla_{D}$ for
$D\in\mathsf{g}$ and $\phi\in\enm_{R}(M).$

\begin{proposition}
Let $\nabla$ be a (not necessesarly integrable) $\mathsf{g}$-connection on
$M.$ \label{end_con}

\begin{enumerate}
\item The induced $\mathsf{g}$-connection $\overline{\nabla}:\mathsf{g}%
\rightarrow\enm_{k}(\overline{R})$ is given by $\overline{\nabla}_{D}%
(\phi)=D(\phi)$ for $\phi\in\overline{R},$ where $\phi$ is identified with an
element in $Q$ and $D$ is extended to $Q.$ In particular, $\overline{\nabla}$
is an integrable connection that is independent of $\nabla.$

\item If $R$ is normal, then $\overline{R}=R,$ and $\overline{\nabla}$ is the
action $\tau:\mathsf{g}\rightarrow$ $\der_{k}(R)$.
\end{enumerate}
\end{proposition}

\begin{proof}
We first prove (1). An element $D\in\mathsf{g}$ has a lifting to a derivation
on $Q$ (which we also denote by $D$). For any connection $\nabla$ on $M$, we
may consider $\nabla_{D}$ and $D$ as maps from $M\subseteq Q$ into $Q.$ Then
$\nabla_{D}-D$ is in $\operatorname*{Hom}_{R}(M,Q)$, so that $\nabla
_{D}=D+\psi_{D}$ for some $\psi_{D}\in Q$. A calculation shows that
\begin{align*}
\overline{\nabla}_{D}(\phi)  &  =(\nabla_{D}\phi-\phi\nabla_{D})=(D+\psi
_{D})\phi-\phi(D+\psi_{D})\\
&  =D\phi-\phi D=D(\phi)
\end{align*}
where we consider $\phi$ as an element in $Q.$ In other words, $\mathsf{g}$
acts on $\overline{R}$ through $\overline{\nabla}$ by extending the action of
$\mathsf{g}$ to $Q.$

To prove (2), note that $\overline{R}$ $\subseteq Q$ is integral over $R$
since $\overline{R}$ is a finitely generated $R$-module. By assumtion, $R$ is
normal, hence $R=\overline{R}.$
\end{proof}

With assumptions as above and with $\overline{\nabla}$ as in the proposition,
we will consider the Lie-Rinehart cohomology groups $\operatorname{H}%
_{\operatorname{Rin}}^{n}(\mathsf{g},\overline{R})=\operatorname{H}%
_{\operatorname{Rin}}^{n}(\mathsf{g},\overline{R},\overline{\nabla})$ of
$(\overline{R},\overline{\nabla}).$ We shall give an interpretation of these
cohomology groups for $n=1$ and $n=2$:

\begin{proposition}
\label{int_class}If $M$ admits a $\mathsf{g}$-connection, then there is a
canonical class
\[
\ic(M)\in\operatorname{H}_{\operatorname{Rin}}^{2}(\mathsf{g},\overline{R})
\]
called the integrability class, such that $\ic(M)=0$ if and only if $M$ admits
an integrable $\mathsf{g}$-connection.
\end{proposition}

\begin{proof}
Let $\nabla$ be a $\mathsf{g}$-connection on $M$ and let $\overline{\nabla}$
be the induced $\mathsf{g}$-connection on $\overline{R}=\enm_{R}(M)$. It
follows from the Bianchi identity
\begin{align*}
(d^{2}(K_{\nabla}))(D_{1}\wedge D_{2}\wedge D_{3})  &  =\overline{\nabla
}_{D_{1}}K_{\nabla}(D_{2}\wedge D_{3})-\overline{\nabla}_{D_{2}}K_{\nabla
}(D_{1}\wedge D_{3})\\
&  +\overline{\nabla}_{D_{3}}K_{\nabla}(D_{1}\wedge D_{2})-K_{\nabla}%
([D_{1},D_{2}]\wedge D_{3})\\
&  +K_{\nabla}([D_{1},D_{3}]\wedge D_{2})-K_{\nabla}([D_{2},D_{3}]\wedge
D_{1})=0
\end{align*}
that $K_{\nabla}$ is a $2$-cocycle in the Rinehart complex $\operatorname{C}%
_{\operatorname{Rin}}^{\ast}(\mathsf{g},\overline{R})$. We define
$\ic(M)=[K_{\nabla}]\in\operatorname{H}_{\operatorname{Rin}}^{2}%
(\mathsf{g},\overline{R})$, and see that $\ic(M)=0$ if and only if $K_{\nabla
}=d^{1}(\tau)$ for some potential $\tau\in\operatorname{C}_{\operatorname{Rin}%
}^{1}(\mathsf{g},\overline{R})$. A calculation shows that this condition holds
if and only if $\nabla-\tau$ is an integrable $\mathsf{g}$-connection on $M$,
since $K_{\nabla^{\prime}}=K_{\nabla}+d^{1}(P)$ when $\nabla^{\prime}%
=\nabla+P.$
\end{proof}

\begin{definition}
Let $\nabla$ and $\nabla^{\prime}$ be two $\mathsf{g}$-connections on $M.$ We
will say that $\nabla$ and $\nabla^{\prime}$ are equivalent if there is an
$R$-linear automorphism $\varphi$ of $M$ such that diagram commutes%
\[%
\begin{array}
[c]{lll}%
M & \overset{\nabla_{D}}{\longrightarrow} & M\\
\downarrow\varphi &  & \downarrow\varphi\\
M & \overset{\nabla_{D}^{\prime}}{\longrightarrow} & M
\end{array}
\]
commutes for all $D\in\mathsf{g}$.
\end{definition}

On may consider the category of modules with $\mathsf{g}$-connections, see
Section 1 in Eriksen and Gustavsen \cite{EriGus:08}. Then $\nabla\sim
\nabla^{\prime}$ if and only if $(M,\nabla)$ and $(M,\nabla^{\prime})$ are
isomorphic in this category.

\begin{theorem}
\label{torsor} Assume that $(R,\mathfrak{m})$ is a reduced complete local
noetherian $k$-algebra with residue field $k,$ and let $(M,\nabla)$ be a rank
one torsion free finitely generated $R$-module with an integrable $\mathsf{g}%
$-connection. Then there is a bijective correspondence between
$\operatorname{H}_{\operatorname{Rin}}^{1}(\mathsf{g},\overline{R})$ and the
set of equivalence classes of integrable $\mathsf{g}$-connections on $M$.
\end{theorem}

\begin{proof}
Let $\tau\in\operatorname{Hom}_{R}(\mathsf{g},\overline{R}).$ We claim that
$\nabla-\tau$ is an integrable $\mathsf{g}$-connection if and only if $\tau$
is a 1-cocycle in $\operatorname{C}_{\operatorname{Rin}}^{\ast}(\mathsf{g}%
,\overline{R}).$ By definition, $d^{1}(\tau)(D_{1}\wedge D_{2})=\overline
{\nabla}_{D_{1}}\tau(D_{2})-\overline{\nabla}_{D_{2}}\tau(D_{1})-\tau
([D_{1},D_{2}])$ and therefore
\[
K_{\nabla-\tau}(D_{1}\wedge D_{2})=K_{\nabla}(D_{1}\wedge D_{2})+d^{1}%
(\tau)(D_{1}\wedge D_{2})+[\tau(D_{1}),\tau(D_{2})].
\]
Since $\overline{R}$ is commutative, $[\tau(D_{1}),\tau(D_{2})]=0$ and this
proves the claim.

The correspondence between $H^{1}(\mathsf{g},\overline{R})$ and equivalence
classes of integrable $\mathsf{g}$-connections is induced by $\tau
\mapsto\nabla-\tau.$ We must show that $\tau\in\operatorname*{im}d^{0}$ if and
only if $\nabla$ and $\nabla-\tau$ are equivalent.

Assume $\tau=d^{0}(\phi)$, where $\phi\in\enm_{R}(M)\cong\operatorname{C}%
_{\operatorname{Rin}}^{0}(\mathsf{g},\overline{R})$. This means that
$\tau(D)=\overline{\nabla}_{D}(\phi)$ for all $D\in\mathsf{g}.$ We claim that
there exists an automorphism $\psi\in\overline{R}$ such that $\psi\nabla
_{D}=(\nabla_{D}-\tau(D))\psi$ for all $D\in\mathsf{g}$. In fact, we have
$\tau(D)=D(\phi)$ and $\nabla_{D}\psi-\psi\nabla_{D}=D(\psi)$ by Proposition
\ref{end_con}, so $\psi\nabla_{D}=(\nabla_{D}-\tau(D))\psi$ if and only if
$D(\psi)=D(\phi)\psi$. Since $\overline{R}$ is is a finitely generated
$R$-module, $\overline{R}/\mathfrak{m}\overline{R}$ is an artinian ring. We
have that $J\cap R=\mathfrak{m}$ where $J$ is the Jacobsen radical in
$\overline{R},$ and it follows that $J^{n}\subseteq\mathfrak{m}\overline{R}$
for some $n.$ Thus $\overline{R}$ is complete in the $J$-adic topology. It
follows that $\overline{R}$ is a product of complete local rings with residue
field $k.$ We have $k^{r}\subset\overline{R}$ and $\overline{R}/J\cong k^{r}$
for some $r.$ If $e$ is any idempotent and $D$ a $k$-linear derivation, it
follows that $D(e)=0.$ Hence $d^{0}(k^{r})=0$ and we may assume that
$\tau=D(\phi)$ for $\phi\in J.$ It follows that $\psi=\exp(\phi)$ is in
$\overline{R}$ and $\psi\nabla_{D}=(\nabla_{D}-\tau(D))\psi$.

Conversely, assume that there is an automorphism $\psi\in\overline{R}$ such
that $\psi\nabla_{D}=(\nabla_{D}-\tau(D))\psi$ for all $D\in\mathsf{g}$. Since
$\psi$ is a unit, we can take $\phi=\log(\psi),$ and by an argument similar to
the one above $\phi\in\overline{R}.$ This implies $\tau=d^{0}(\phi)$.
\end{proof}

\section{The curve case}

In this section we assume that $R$ is reduced noetherian $k$-algebra of
dimension one, and consider in some detail the case when $R$ is a monomial curve.

\begin{proposition}
Let $\mathsf{g}\subseteq\der_{k}(R)$ be a Lie-Rinehart algebra and let $M$ be
a torsion free $R$-module. Then any $\mathsf{g}$-connection on $M$ is integrable.
\end{proposition}

\begin{proof}
If $\dim R=1,$ one has that $\operatorname{Hom}\nolimits_{R}(\wedge
^{2}\mathsf{g},\operatorname{End}_{R}(M))=0$ when $\mathsf{g}\subset
\operatorname{Der}_{k}(R)$ since $\enm_{R}(M)$ is torsion free and $\wedge
^{2}\mathsf{g}\subseteq\wedge^{2}\operatorname{Der}_{k}(R)$ is a torsion
module. If there exists a connection $\nabla$ on $M,$ the curvature, which is
an $R$-linear map $K_{\nabla}:\mathsf{g}\wedge\mathsf{g}\rightarrow
\enm_{R}(M)$, is necessarily zero. Thus the connection is automatically integrable.
\end{proof}

Let $R=k[\Gamma]$ be a monomial curve singularity given by a numerical
semigroup $\Gamma\subseteq\mathbb{N}_{0}$, where $H=\mathbb{N}_{0}%
\setminus\Gamma$ is a finite set. We consider a finitely generated graded
torsion free $R$-module $M$ of rank one. Up to graded isomorphism and a shift,
we may assume that $M$ has the form $M=k[\Lambda]$, where $\Lambda$ is a set
such that $\Gamma\subseteq\Lambda\subseteq\mathbb{N}_{0}$ and $\Gamma
+\Lambda\subseteq\Lambda$. Let $\Gamma^{(1)}=\{w\in H:w+(\Gamma\setminus
\{0\})\subseteq\Gamma\}$ and let $E$ be the Euler derivation. Then the set
$\{E\}\cup\{t^{w}E:w\in\Gamma^{(1)}\}$ is a minimal generating set for
$\mathsf{g}=\der_{k}(R)$ as a left $R$-module, see Eriksen \cite{eri:03}.

\begin{theorem}
Let $R=k[\Gamma]$ be a monomial curve singularity given by a numerical
semigroup $\Gamma\subseteq\mathbb{N}_{0}$, where $\mathbb{N}_{0}%
\setminus\Gamma$ is a finite set, and let $M$ be a finitely generated graded
torsion free $R$-module of rank one. If $\nabla$ is a connection on $M$, then
\[
\operatorname{H}_{\operatorname{Rin}}^{i}(\operatorname{Der}_{k}%
(R),M,\nabla)=0\text{ for }i\geq1.
\]

\end{theorem}

\begin{proof}
Let $\mathsf{g}=\operatorname{Der}_{k}(R)$. Consider the set $S=\{\lambda
\in\Lambda:\lambda+\Gamma^{(1)}\not \subseteq \Lambda\}$ and let $l$ be the
cardinality of $S$. There are three possibilities:

\begin{enumerate}
\item $l = 0$: $\nabla_{E} = E - c$ defines an integrable connection on $M$
for all $c \in k$.

\item $l = 1$: $\nabla_{E} = E - c$ defines an integrable connection on $M$
iff $c = \lambda_{0}$ is the unique element in $S$.

\item $l \ge2$: there are no connections on $M$
\end{enumerate}

Assume $l\leq1$ and consider the connection given by $\nabla_{E}=E-c$, where
$c\in k$ if $l=0$ and $c=\lambda_{0}$ if $l=1$. Let $r:\operatorname{Hom}%
_{R}(\mathsf{g},M)\rightarrow M$ be given by $\phi\mapsto\phi(E)$. Then the
composition
\[
M\xrightarrow{d^{0}}\operatorname{Hom}\nolimits_{R}(\mathsf{g}%
,M)\xrightarrow{r} M
\]
is the operator $\nabla_{E}$. We claim that $r$ is injective with image
$k[\Lambda\setminus\{c\}]\subseteq M$. In fact, if $\phi\in\operatorname{Hom}%
_{R}(\mathsf{g},M)$ with $\phi(E)=0$, then $\phi(t^{w}E)=t^{w}\phi(E)=0$ since
$M$ is torsion free. Therefore $r$ is injective. Consider $t^{\lambda}$ with
$\lambda\in\Lambda\setminus\{c\}$. Since $\lambda\not \in S$, we have that
$\Gamma^{(1)}+\lambda\subseteq\Lambda$. Therefore, $\phi(E)=t^{\lambda}$ and
$\phi(t^{w}E)=t^{w}t^{\lambda}$ for $w\in\Gamma^{(1)}$ defines a well-defined
$R$-linear map $\phi:\mathsf{g}\rightarrow M$. Moreover, we clearly have
$t^{c}\not \in \operatorname{im}(r)$. Since $r$ is a graded homomorphism, this
proves that $\operatorname{im}(r)=k[\Lambda\setminus\{c\}]\subseteq M.$ If
$\phi\in\hmm_{R}(\mathsf{g},M)$ is homogenous of degree $w$, then $\phi(E)\in
M$ is homogenous of degree $w$ and $w\ne c.$ This implies that $d^{0}%
(\phi(E)/(w-c))=\phi$ and therefore $d^{0}$ is surjective. We conclude that
$\operatorname{H}_{\operatorname{Rin}}^{1}(\mathsf{g},M)=0$. Since
$\operatorname{C}^{i}_{\operatorname{Rin}}(\mathsf{g},M,\nabla)=\hmm_{R}%
(\wedge^{i}\mathsf{g},M)=0$ for $i\ge2,$ $\operatorname{H}_{\operatorname{Rin}%
}^{i}(\operatorname{Der}_{k}(R),M,\nabla)=0\text{ for }i\geq1.$
\end{proof}

Given a graded torsion free $R$-module $M$ of rank one on a monomial curve
$R,$ it does not necessarily exist a connection on $M,$ see Section 5.2 in
Eriksen and Gustavsen \cite{EriGus:08}. However, if there exists a connection,
it is integrable and unique up to analytic isomorphism by Theorem \ref{torsor}
and the proposition above.

\section{The case of an isolated normal singularity}

In this section, we assume that $R$ is a noetherian Cohen-Macaulay domain over
$k$ of dimension $d\geq2$ with a unique isolated singularity. For any finitely
generated torsion free $R$-module $M$ of rank one, $\overline{R}%
=\enm_{R}(M)=R$ from Proposition \ref{end_con} since $R$ is normal by Serre's
normality criterion. Moreover, if $M$ admits a connection $\nabla,$ then the
induced connection $\overline{\nabla}$ on $\overline{R}$ is the standard
action of $\operatorname{Der}_{k}(R)$ on $R.$

From Proposition \ref{int_class}, we know that $\operatorname{H}%
_{\operatorname{Rin}}^{2}(\operatorname{Der}_{k}(R),R)$ contains the
obstruction for the existence of an integrable connection on $M$. If it
vanishes, then it follows from Theorem \ref{torsor} that $\operatorname{H}%
_{\operatorname{Rin}}^{1}(\operatorname{Der}_{k}(R),R)$ is a moduli space for
the integrable connections on $M,$ up to analytic equivalence.

When $k=\mathbb{C}$ is the complex numbers, there are also other
interpretations of the $k$-vector spaces $\operatorname{H}_{\operatorname{Rin}%
}^{1}(\operatorname{Der}_{k}(R),R)$ and $\operatorname{H}_{\operatorname{Rin}%
}^{2}(\operatorname{Der}_{k}(R),R)$: In this section, we will relate the
Lie-Rinehart cohomology $\operatorname{H}_{\operatorname{Rin}}^{\ast
}(\operatorname{Der}_{k}(R),R)$ to the topological cohomology
$\operatorname{H}^{\ast}(U_{\operatorname{an}},\mathbb{C})$ where
$U=X\setminus\{x\},$ $X=\operatorname{Spec}(R)$ and $x\in X$ is the singular
point. The Lie-Rinehart cohomology $\operatorname{H}_{\operatorname{Rin}%
}^{\ast}(\operatorname{Der}_{k}(R),R)$ is also closely related to the
punctured deRham cohomology of Huang, Luk and Yau \cite{HuaLukYau:03} and
therefore to the $\mu$ and $\tau$ invariants.

\begin{theorem}
\label{five_term}Let $R$ be a finitely generated Cohen-Macaulay domain over
$\mathbb{C}$ of dimension $d\geq2$ with a unique isolated singularity $x\in
X=\operatorname{Spec}(R).$ Then there is a natural exact sequence%
\[
0\rightarrow\operatorname{H}_{\operatorname{Rin}}^{1}(\operatorname{Der}%
(R),R)\rightarrow\operatorname{H}^{1}(U_{\operatorname{an}},\mathbb{C}%
)\rightarrow E_{2}^{0,1}\rightarrow\operatorname{H}_{\operatorname{Rin}}%
^{2}(\operatorname{Der}(R),R)\rightarrow\operatorname{H}^{2}%
(U_{\operatorname{an}},\mathbb{C})
\]
where $E_{2}^{0,1}=\ker(\operatorname{H}^{1}(U,\mathcal{O}_{X})\rightarrow
\operatorname{H}^{1}(U,\Omega_{X}^{1}))$ and $U=X\setminus\{x\}$. Moreover, if
$d\geq3,$ then $\operatorname{H}_{\operatorname{Rin}}^{1}(\operatorname{Der}%
(R),R)\cong\operatorname{H}^{1}(U_{\operatorname{an}},\mathbb{C})$ and
$\operatorname{H}_{\operatorname{Rin}}^{2}(\operatorname{Der}%
(R),R)\hookrightarrow\operatorname{H}^{2}(U_{\operatorname{an}},\mathbb{C}).$
\end{theorem}

\begin{proof}
The definition of the Lie-Rinehart complex generalizes to give a complex
$\mathcal{C}_{\operatorname{Rin}}^{\ast}(\Theta_{X},\mathcal{O}_{X})$ of
sheaves on $X,$ given by
\[
\mathcal{C}_{\operatorname{Rin}}^{n}(\Theta_{X},\mathcal{O}_{X})=\mathcal{H}%
om_{\mathcal{O}_{X}}(\wedge^{n}\Theta_{X},\mathcal{O}_{X})
\]
with the natural action of the tangent sheaf $\Theta_{X}$ on $\mathcal{O}%
_{X}.$ In particular, there is a restricted complex $\mathcal{C}%
_{\operatorname{Rin}}^{\ast}|_{U}=\mathcal{C}_{\operatorname{Rin}}^{\ast
}(\Theta_{U},\mathcal{O}_{U})$ of sheaves on $U.$ Denote by $\mathbb{H}%
^{i}=\mathbb{H}^{i}(U,\mathcal{C}_{\operatorname{Rin}}^{\ast}|_{U})$ the
hypercohomology of the sheafified Lie-Rinehart complex, see for instance 5.7.9
in Weibel \cite{Wei:94}. From the five term sequence, we get%
\[
0\rightarrow E_{2}^{1,0}\rightarrow\mathbb{H}^{1}\rightarrow E_{2}%
^{0,1}\rightarrow E_{2}^{2,0}\rightarrow\mathbb{H}^{2}%
\]
where $E_{2}^{p,q}={}^{I}\hspace{-2pt}E_{2}^{p,q}\cong\operatorname{H}%
^{p}(\operatorname{H}^{q}(U,\mathcal{C}_{\operatorname{Rin}}^{\ast}|_{U})).$
Consider in particular the vector spaces $E_{2}^{p,0}=\operatorname{H}%
^{p}(\operatorname{H}^{0}(U,\mathcal{C}_{\operatorname{Rin}}^{\ast}|_{U})).$
Since $\mathcal{C}_{\operatorname{Rin}}^{p}|_{U}$ is sheaf of reflexive
modules for $p\geq0$ by Corollary 1.2 in Hartshorne \cite{Har:80}, we get from
Proposition 1.6.(iii) in \cite{Har:80} that $\operatorname{H}^{0}%
(U,\mathcal{C}_{\operatorname{Rin}}^{\ast}|_{U})=$ $\operatorname*{C}%
_{\operatorname{Rin}}^{\ast}(\operatorname{Der}(R),R).$ Thus $E_{2}^{p,0}%
\cong\operatorname{H}_{\operatorname{Rin}}^{p}(\operatorname{Der}(R),R).$ Note
further that since $U$ is smooth, $\mathcal{C}_{\operatorname{Rin}}^{\ast
}|_{U}$ coincides with the deRham complex, so by Grothendieck's algebraic
deRham theorem, $\mathbb{H}^{i}=\mathbb{H}^{i}(U,\mathcal{C}%
_{\operatorname{Rin}}^{\ast}|_{U})\cong\operatorname{H}^{i}%
(U_{\operatorname{an}},\mathbb{C}),$ see \cite{Gro:66a}. On the other hand, we
see that
\[
E_{2}^{0,1}=\operatorname{H}^{0}(\operatorname{H}^{1}(U,\mathcal{C}%
_{\operatorname{Rin}}^{\ast}|_{U}))\cong\ker(\operatorname{H}^{1}%
(U,\mathcal{O}_{X})\rightarrow\operatorname{H}^{1}(U,\Omega_{X}^{1})).
\]

For the last part, we notice that $\operatorname{H}^{1}(U,\mathcal{O}%
_{X})=\operatorname{H}_{\{x\}}^{2}(\mathcal{O}_{X}),$ where the last group is
the local cohomology with respect to the closed subscheme $\{x\},$ see for
instance 4.6.2 in Weibel \cite{Wei:94}. By Corollary 4.6.9 in \cite{Wei:94},
this group vanishes if $d\geq3.$ In particular, it follows that $E_{2}%
^{0,1}=0$ in this case.
\end{proof}

For surface singularities, it is in general difficult to compute $E_{2}^{0,1}$
directly. We have the following partial results:

\begin{remark}
For a rational complex surface singularity with link $L,$ it is known that
$\operatorname{H}^{1}(L,\mathbb{C})=0,$ see Mumford \cite{Mum:61}, and by
Poincar\`{e} duality, $\operatorname{H}^{2}(L,\mathbb{C})=0.$ For simplicity,
we assume that $R$ is quasi-homogenous (for instance a quotient singularity).
In this case, we have $\operatorname{H}^{i}(U_{\operatorname{an}}%
,\mathbb{C})=\operatorname{H}^{i}(L,\mathbb{C})=0$ for $i=1$ and $2.$ Thus
$\operatorname{H}_{\operatorname{Rin}}^{1}(\operatorname{Der}(R),R)=0\ $and
$E_{2}^{0,1}\cong\operatorname{H}_{\operatorname{Rin}}^{2}(\operatorname{Der}%
(R),R).$
\end{remark}

\begin{remark}
For simple elliptic complex surface singularities, Kahn has shown that
$\operatorname{H}_{\operatorname{Rin}}^{1}(\operatorname{Der}(R),R)\cong%
\mathbb{C}$, see \cite{kah:88}. One the other hand, $\operatorname{H}%
^{1}(U_{\operatorname{an}},\mathbb{C})\cong\mathbb{C}^{2},$ so $E_{2}%
^{0,1}\neq0$ in this case.
\end{remark}

Let us consider the punctured local holomorphic de Rham cohomology $H_{h}%
^{i}(V,x)$ introduced in Huang, Luk and Yau \cite{HuaLukYau:03} for a germ
$(V,x)$ of a complex analytic space. In fact, for a hypersurface singularity
$R$ of dimension $d\geq2,$ $H_{h}^{i}(V,x)$ coincides with $\operatorname{H}%
_{\operatorname{Rin}}^{i}(\operatorname{Der}_{\mathbb{C}}(R),R)$ for $i\geq1,$
see Lemma 2.7 in \cite{HuaLukYau:03} and Proposition 1.6.(iii) in Hartshorne
\cite{Har:80}. The following result follows from the main theorem in Huang,
Luk and Yau \cite{HuaLukYau:03}:

\begin{theorem}
Assume that $R=\mathbb{C}[[x_{0},\dots,x_{d}]]/(f)$ and that $R$ is an
isolated singularity of dimension $d\geq2$. Then

\begin{enumerate}
\label{Huang_etal}

\item $\dim_{\mathbb{C}}\operatorname{H}_{\operatorname{Rin}}^{i}%
(\operatorname{Der}_{\mathbb{C}}(R),R)=0$ for $1\leq i\leq d-2,$

\item $\dim_{\mathbb{C}}\operatorname{H}_{\operatorname{Rin}}^{d}%
(\operatorname{Der}_{\mathbb{C}}(R),R)-\dim_{\mathbb{C}}\operatorname{H}%
_{\operatorname{Rin}}^{d-1}(\operatorname{Der}_{\mathbb{C}}(R),R)=\mu-\tau$,
\end{enumerate}

\noindent where $\mu$ is the Milnor number and $\tau$ is the Tjurina number.
\end{theorem}

\begin{corollary}
If $d\geq3,$ then $\operatorname{H}_{\operatorname{Rin}}^{1}%
(\operatorname{Der}_{\mathbb{C}}(R),R)\cong\operatorname{H}^{1}%
(U_{\operatorname{an}},\mathbb{C})=0$ and if $d\geq4$, then $\operatorname{H}%
_{\operatorname{Rin}}^{2}(\operatorname{Der}_{\mathbb{C}}(R),R)=0$.
\end{corollary}

\section{The case of a quasi-homogenous surface}

In this section, we compute the Lie-Rinehart cohomology $\operatorname{H}%
_{\operatorname{Rin}}^{\ast}(\operatorname{Der}_{k}(R),R)$ in the case of a
integral quasi-homogenous surface singularity $R=k[x_{1},x_{2},x_{3}]/(f).$ We
write $d_{i}=\deg x_{i}$ for $i=1,2,3$, $d=\deg f\geq2$ and put $\omega
_{i}=d_{i}/d$, $\delta=\omega_{1}+\omega_{2}+\omega_{3}-1.$

The Lie-Rinehart complex $\operatorname*{C}\nolimits^{\ast}=\operatorname*{C}%
\nolimits_{\operatorname*{Rin}}^{\ast}(\operatorname{Der}_{k}(R),R)$ in the
present case is given as%
\[
\operatorname*{C}\nolimits^{0}=R\overset{d^{0}}{\longrightarrow}%
\operatorname*{C}\nolimits^{1}=\operatorname{Hom}_{R}(\operatorname{Der}%
_{k}(R),R)\overset{d^{1}}{\longrightarrow}\operatorname*{C}\nolimits^{2}%
=\operatorname{Hom}_{R}(\wedge^{2}\operatorname{Der}_{k}(R),R)\rightarrow0
\]
since $\wedge^{3}\operatorname{Der}_{k}(R)$ is supported at the singular locus
of $\operatorname{Spec}(R).$ The map $d^{0}$ is given by $d^{0}(r)(D)=D(r)$
for $r\in R,$ $D\in\operatorname{Der}_{k}(R)$ and $d^{1}$ is given by%
\[
d^{1}(\varphi)((D_{1}\wedge D_{2}))=D_{1}(\varphi(D_{2}))-D_{2}(\varphi
(D_{1}))-\varphi([D_{1},D_{2}])
\]
for $\varphi\in\operatorname{Hom}_{R}(\operatorname{Der}_{k}(R),R)$ and
$D_{1},D_{2}\in\operatorname{Der}_{k}(R),$ Note that $\operatorname{H}%
_{\operatorname{Rin}}^{i}(\operatorname{Der}_{k}(R),R)=0$ for $i\geq3.$

It is clear that $\operatorname*{C}\nolimits^{1}=\operatorname{Hom}%
_{R}(\operatorname{Der}_{k}(R),R)$ and $\operatorname*{C}\nolimits^{2}%
=\operatorname{Hom}_{R}(\wedge^{2}\operatorname{Der}_{k}(R),R)$ are graded,
and that $d^{0}$ and $d^{1}$ are homogenous of degree zero. $\ $It is further
well-known that $\operatorname{Der}_{k}(R)$ is naturally generated by the
Euler derivation $E$ (homogenous of degree $0$) and the Kozul derivations
$D_{1},$ $D_{2},$ $D_{3}$ (homogenous of degree $d-d_{1}-d_{2},$
$d-d_{1}-d_{3},$ $d-d_{2}-d_{3}$ respectively), given by%
\begin{align*}
E  &  =\omega_{1}x_{1}\frac{\partial}{\partial x_{1}}+\omega_{2}x_{2}%
\frac{\partial}{\partial x_{2}}+\omega_{3}x_{3}\frac{\partial}{\partial x_{3}}
& D_{1}  &  =\frac{\partial f}{\partial x_{2}}\frac{\partial}{\partial x_{1}%
}-\frac{\partial f}{\partial x_{1}}\frac{\partial}{\partial x_{2}},\\
D_{2}  &  =\frac{\partial f}{\partial x_{3}}\frac{\partial}{\partial x_{1}%
}-\frac{\partial f}{\partial x_{1}}\frac{\partial}{\partial x_{3}}, & D_{3}
&  =\frac{\partial f}{\partial x_{3}}\frac{\partial}{\partial x_{2}}%
-\frac{\partial f}{\partial x_{2}}\frac{\partial}{\partial x_{3}}%
\end{align*}
To give a description of $C^{1}=\operatorname{Hom}_{R}(\operatorname{Der}%
_{k}(R),R),$ we define%
\[
\varphi=\left(
\begin{array}
[c]{cccc}%
f_{1} & f_{2} & f_{3} & 0\\
\omega_{2}x_{2} & -\omega_{1}x_{1} & 0 & f_{3}\\
\omega_{3}x_{3} & 0 & -\omega_{1}x_{1} & -f_{2}\\
0 & \omega_{3}x_{3} & -\omega_{2}x_{2} & f_{1}%
\end{array}
\right)  \text{ and }\psi=\left(
\begin{array}
[c]{cccc}%
\omega_{1}x_{1} & f_{2} & f_{3} & 0\\
\omega_{2}x_{2} & -f_{1} & 0 & f_{3}\\
\omega_{3}x_{3} & 0 & -f_{1} & -f_{2}\\
0 & \omega_{3}x_{3} & -\omega_{2}x_{2} & \omega_{1}x_{1}%
\end{array}
\right)
\]
where $f_{i}=\partial f/\partial x_{i}.$

\begin{lemma}
The matrices $(\varphi,\psi)$ give a matrix factorization of
$\operatorname{Der}_{k}(R)$ and the transposed matrices $(\varphi^{T},\psi
^{T})$ give a matrix factorization of $\operatorname{Hom}_{R}%
(\operatorname{Der}_{k}(R),R).$
\end{lemma}

\begin{proof}
This follows from Lemma 1.5 in Yoshino and Kawamoto \cite{YosKaw:88} and
Proposition 2.1 in Behnke \cite{Beh:89}. For the last part see for instance
Lemma 11 in Eriksen and Gustavsen \cite{EriGus:08}.
\end{proof}

Mapping $\operatorname{Hom}_{R}(\operatorname{Der}_{k}(R),R)$ into $R^{4}$ by
evaluation on $(E,D_{1},D_{2},D_{3}),$ we obtain the rows $\psi^{(i)}$ in
$\psi$ as generators for $\operatorname{Hom}_{R}(\operatorname{Der}(R),R)$ in
$R^{4}.$ We see that $\deg\psi^{(i)}=d_{i}$ for $i=1,2,3,$ and $\deg\psi
^{(4)}=d_{1}+d_{2}+d_{3}-d=d\delta.$

To give a description of $\operatorname*{C}\nolimits^{2}=\operatorname{Hom}%
_{R}(\wedge^{2}\operatorname{Der}_{k}(R),R),$ we consider the element
$\Delta=\frac{\partial f}{\partial x_{3}}\frac{\partial}{\partial x_{1}}%
\wedge\frac{\partial}{\partial x_{2}}-\frac{\partial f}{\partial x_{2}}%
\frac{\partial}{\partial x_{1}}\wedge\frac{\partial}{\partial x_{3}}%
+\frac{\partial f}{\partial x_{1}}\frac{\partial}{\partial x_{2}}\wedge
\frac{\partial}{\partial x_{3}}$ of degree $d-d_{1}-d_{2}-d_{3}=-d\delta.$ A
calculation gives%
\begin{align*}
E\wedge D_{1}  &  =\omega_{3}x_{3}\Delta, & E\wedge D_{2}  &  =-\omega
_{2}x_{2}\Delta, & E\wedge D_{3}  &  =\omega_{1}x_{1}\Delta,\\
D_{1}\wedge D_{2}  &  =\frac{\partial f}{\partial x_{1}}\Delta, & D_{1}\wedge
D_{3}  &  =\frac{\partial f}{\partial x_{2}}\Delta, & D_{2}\wedge D_{3}  &
=\frac{\partial f}{\partial x_{3}}\Delta,
\end{align*}
and we conclude that $\wedge^{2}\operatorname{Der}_{k}(R)=(x_{1},x_{2}%
,x_{3})\Delta.$ From this, we get the following isomorphisms of graded
modules:%
\[
\operatorname{Hom}_{R}(\wedge^{2}\operatorname{Der}_{k}%
(R),R)=\operatorname{Hom}_{R}(\mathfrak{m}\Delta,R)\cong\operatorname{Hom}%
_{R}(R\Delta,R)\cong R[-\deg\Delta].
\]
We compute the map $d^{1}$ and get%

\begin{align*}
d^{1}(r\psi^{(1)})(E\wedge D_{3}) &  =E(r\psi^{(1)})-D_{3}(r\psi
^{(1)}(E))-r\psi^{(1)}([E,D_{3}])=-\omega_{1}x_{1}D_{3}(r)\\
d^{1}(r\psi^{(2)})(E\wedge D_{2}) &  =E(r\psi^{(2)})-D_{2}(r\psi
^{(2)})(E)-r\psi^{(2)}([E,D_{2}])=-\omega_{2}x_{2}D_{2}(r)\\
d^{1}(r\psi^{(3)})(E\wedge D_{1}) &  =E(r\psi^{(3)})-D_{1}(r\psi
^{(3)})(E)-r\psi^{(3)}([E,D_{1}])=-\omega_{3}x_{3}D_{1}(r)\\
d^{1}(r\psi^{(4)})(E\wedge D_{1}) &  =E(r\psi^{(4)})-D_{1}(r\psi
^{(4)})(E)-r\psi^{(4)}([E,D_{1}])=\omega_{3}x_{3}(E(r)+\delta r)
\end{align*}
using%
\[
\lbrack E,D_{1}]=(1-\omega_{1}-\omega_{2})D_{1},\text{ \ }[E,D_{2}%
]=(1-\omega_{1}-\omega_{3})D_{2},\text{ }[E,D_{3}]=(1-\omega_{2}-\omega
_{3})D_{3}%
\]
From this we conclude that%
\begin{align*}
d^{1}(r\psi^{(1)})(\Delta) &  =-D_{3}(r) & d^{1}(r\psi^{(2)})(\Delta) &
=D_{2}(r)\\
d^{1}(r\psi^{(3)})(\Delta) &  =-D_{1}(r) & d^{1}(r\psi^{(4)})(\Delta) &
=E(r)+\delta r,
\end{align*}
and in conclusion we have reached a very concrete description of the
Lie-Rinehart complex $\operatorname*{C}\nolimits_{\operatorname*{Rin}}^{\ast
}(\operatorname{Der}_{k}(R),R).$ Using this, we are able to prove the
following result:

\begin{theorem}
\label{qh_case}Let $R=k[x_{1},x_{2},x_{3}]/(f)$ be an integral
quasi-homogenous surface singularity. Then the grading on $R$ induces a
grading on $\operatorname{H}_{\operatorname{Rin}}^{\ast}(\operatorname{Der}%
_{k}(R),R),$ and
\begin{align*}
\operatorname{H}_{\operatorname{Rin}}^{0}(\operatorname{Der}_{k}(R),R)  &
=\operatorname{H}_{\operatorname{Rin}}^{0}(\operatorname{Der}_{k}%
(R),R)_{0}\cong R_{0}=k\\
\operatorname{H}_{\operatorname{Rin}}^{i}(\operatorname{Der}_{k}(R),R)  &
=\operatorname{H}_{\operatorname{Rin}}^{i}(\operatorname{Der}_{k}%
(R),R)_{0}\cong R_{d-d_{1}-d_{2}-d_{3}}\text{ for }i=1,2\\
\operatorname{H}_{\operatorname{Rin}}^{i}(\operatorname{Der}_{k}(R),R)  &
=0\text{ for }i\geq3,
\end{align*}
as graded $k$-vector space, where $d_{i}=\deg x_{i}$ for $i=1,2,3$ and $d=\deg
f\geq2.$
\end{theorem}

\begin{proof}
For simplicity, we write $\operatorname{H}^{i}=\operatorname{H}%
_{\operatorname{Rin}}^{i}(\operatorname{Der}_{k}(R),R)$ for $i\geq0.$ We have
that $\operatorname{H}^{0}=\{r\in R:D(r)=0\;\forall D\in\operatorname{Der}%
_{k}(R)\}.$ If $r\in\operatorname{H}_{\omega}^{0},$ then $E(r)=\omega r=0$
implies that $\omega=0$ or $r=0.$ Thus $\operatorname{H}^{0}=R_{0}=k.$

We have that $\operatorname{H}^{2}=\operatorname{coker}d^{1}=$
$\operatorname*{C}\nolimits^{2}/\operatorname{im}d^{1}\cong R[-d\delta
]/\operatorname{im}d^{1}$ and $\operatorname{im}d^{1}$ is spanned by
$D_{i}(r)$ for $i=1,2,3$ and $E(r)+\delta r=(\frac{\deg r}{d}+\delta)r$ as $r$
runs through all homogenous elements in $R.$ Since $\deg D_{i}=-d\delta
+d_{i},$ for $i=1,2,3,$ it follows that $D_{i}(r)\in R_{-d\delta+d_{i}+\deg
r}=\operatorname*{C}\nolimits_{{}_{d_{i}+\deg r}}^{2}$ for $i=1,2,3$.
Furthermore, $\frac{\deg r}{d}+\delta=0$ if and only if $\deg r=-\delta d$. We
conclude that $\operatorname{im}d^{1}=\operatorname*{C}\nolimits_{\neq0}^{2}$,
and hence $\operatorname{H}^{2}=\operatorname*{C}\nolimits_{0}^{2}\cong
R_{d-d_{1}-d_{2}-d_{2}}.$

To compute $\operatorname{H}^{1},$ we note that $\operatorname{im}d_{0}^{0}=0$
and $\ker d_{0}^{1}\cong R_{-d\delta}\cdot\psi^{(4)}$ by the argument above.
It follows that $\operatorname{H}_{0}^{1}\cong R_{-d\delta}\cdot\psi
^{(4)}\cong R_{d-d_{1}-d_{2}-d_{2}}.$ We claim that $\operatorname{H}_{\omega
}^{1}=0$ for $\omega\neq0.$ To prove the claim, we first note that since
$\operatorname{H}_{\omega}^{2}=0$ for $\omega\neq0,$ it follows that
$d_{\omega}^{1}$ induces an isomorphism $\operatorname*{C}\nolimits_{\omega
}^{1}/\ker d_{\omega}^{1}\cong\operatorname*{C}\nolimits_{\omega}^{2}$ for
$\omega\neq0$. Also, $\operatorname{im}d_{\omega}^{0}\cong\operatorname*{C}%
\nolimits_{\omega}^{0}$ for $\omega\neq0.$ Thus%
\[
\dim_{k}\operatorname{H}_{\omega}^{1}=\dim_{k}\ker d_{\omega}^{1}-\dim
_{k}\operatorname{im}d_{\omega}^{0}=\dim_{k}\operatorname*{C}\nolimits_{\omega
}^{1}-\dim_{k}\operatorname*{C}\nolimits_{\omega}^{2}-\dim_{k}%
\operatorname*{C}\nolimits_{\omega}^{0}%
\]
for $\omega\neq0.$ To compute these dimensions, recall the Auslander sequence
\[
0\rightarrow\omega_{R}\rightarrow E\rightarrow\mathfrak{m}\rightarrow0
\]
where $\omega_{R}$ is the canonical module, $E$ is the fundamental module and
$\mathfrak{m}$ is the maximal graded ideal of $R.$ From (the proof of)
Proposition 2.1 in Behnke \cite{Beh:89}, we have that $\omega_{R}%
\cong\operatorname{Hom}_{R}(\wedge^{2}\operatorname{Der}_{k}%
(R),R)=\operatorname*{C}\nolimits^{2}$ and $E\cong\operatorname{Hom}%
_{R}(\operatorname{Der}_{k}(R),R)=\operatorname*{C}\nolimits^{1}$ as graded
modules, since $R$ is quasi-homogenous, see also Lemma 1.2 in Yoshino and
Kawamoto \cite{YosKaw:88}. Since there are homogenous isomorphisms
$\operatorname{Ext}_{R}^{1}(\mathfrak{m},\omega_{R})\cong\operatorname{Ext}%
_{R}^{2}(R/\mathfrak{m},\omega_{R})\cong R/\mathfrak{m}$ of degree zero, see
Definition 3.6.8, Example 3.6.10 and Proposition 3.6.12 in Bruns and Herzog
\cite{BruHer:93}, it follows that the Auslander sequence is homogenous of
degree zero. Thus $\dim\operatorname*{C}\nolimits_{\omega}^{1}=\dim
\operatorname*{C}\nolimits_{\omega}^{2}+\dim\operatorname*{C}\nolimits_{\omega
}^{0}$ for $\omega\neq0.$ This proves the claim that $\operatorname{H}%
_{\omega}^{1}=0$ for $\omega\neq0.$
\end{proof}

\begin{remark}
It follows from Theorem \ref{Huang_etal} that $\operatorname{H}%
_{\operatorname{Rin}}^{1}(\operatorname{Der}_{k}(R),R)\cong\operatorname{H}%
_{\operatorname{Rin}}^{2}(\operatorname{Der}_{k}(R),R)$ when $R\cong
k[x_{1},x_{2}x_{3}]/(f)$ is a quasi-homogenous surface singularity, since it
is known that $\mu=\tau$ in the quasi-homogenous case. From Theorem
\ref{qh_case} it follows that $\operatorname{H}_{\operatorname{Rin}}%
^{1}(\operatorname{Der}_{k}(R),R)\cong\operatorname{H}_{\operatorname{Rin}%
}^{2}(\operatorname{Der}_{k}(R),R)$ as graded $k$-vector spaces.
\end{remark}

\begin{remark}
We see that all cohomology is consentrated in degree $0$ in the case covered
by the theorem. It follows from the proof of Proposition \ref{int_class} that
integrability class $\ic(M)$ lies in $\operatorname{H}_{\operatorname{Rin}%
}^{2}(\operatorname{Der}_{k}(R),R)_{0}$ for any graded torsion free rank one
module $M.$ Moreover, if $\ic(M)=0,$ it follows form the proof of Theorem
\ref{torsor} and Theorem \ref{qh_case} that $\operatorname{H}%
_{\operatorname{Rin}}^{1}(\operatorname{Der}_{k}(R),R)_{0}$ is a moduli space
for integrable connections. Hence up to analytic equivalence, all integrable
connections are homogenous.
\end{remark}

\begin{example}
We consider the singularity $R=\mathbb{C}[x_{1},x_{2},x_{3}]/(f),$ where $f$
is homogenous of degree $d,$ with $d_{1}=d_{2}=d_{3}=1.$ Then
$\operatorname{H}_{\operatorname{Rin}}^{i}(\operatorname{Der}_{k}(R),R)\cong
R_{d-1-1-1}=R_{d-3}$ so that%
\[
\dim_{\mathbb{C}}\operatorname{H}_{\operatorname{Rin}}^{i}(\operatorname{Der}%
_{k}(R),R)=\binom{d-3+2}{d-3}=\binom{d-1}{d-3}=\frac{(d-1)(d-2)}{2}%
\]
for $i=1,2.$ This number is the genus of the curve $V(f)$ in $\mathbb{P}^{2},$
and therefore also the genus the exceptional curve in the minimal resolution
of $R.$
\end{example}

\begin{example}
The minimally elliptic singularity $\mathbb{C}[x_{1},x_{2},x_{3}]/(x_{1}%
^{3}+x_{2}^{4}+x_{3}^{4})$ has $d=12,d_{1}=4,d_{2}=d_{3}=3,$ so
$\operatorname{H}_{\operatorname{Rin}}^{i}(\operatorname{Der}_{k}(R),R)\cong
R_{12-4-3-3}=R_{2}=0$ for $i=1,2.$
\end{example}

\begin{corollary}
Let $R=k[x_{1},x_{2},x_{3}]/(f)$ be an integral quasi-homogeneous surface
singularity, and let $M$ be any finitely generated torsion free graded
$R$-module of rank one. Then any homogenous connection on $M$ is integrable.
\end{corollary}

\begin{proof}
Let $\nabla$ be an arbitrary homogenous connection on $M$, and let
\[
0\leftarrow M\leftarrow L_{0}\xleftarrow{d_0}L_{1}%
\]
be a graded presentation of $M$, where $\{e_{i}\}$ and $\{f_{i}\}$ are
homogeneous bases of $L_{0}$ and $L_{1}$, and $d_{0}=(a_{ij})$ is the matrix
of $d_{0}$ with respect to these bases. Then we have $\deg(a_{ij})=\deg
(f_{j})-\deg(e_{i})$ for all $i,j$. We consider the diagonal matrix $P$ with
entries $\epsilon_{j}=(\deg(e_{j})-\deg(e_{1}))/d$ on the diagonal. Since we
have
\[
E(d_{0})=\frac{1}{d}\deg(a_{ij})(a_{ij})=\frac{1}{d}(\deg(f_{j})-\deg
(e_{i}))(a_{ij}),
\]
we see that $E(d_{0})+Pd_{0}=d_{0}Q$ for some $Q\in\enm_{R}(L_{1}).$
Therefore, $\nabla_{E}^{\prime}=E+P\in\enm_{k}(L_{0})$ induces an operator
$\nabla_{E}^{\prime}\in\enm_{k}(M)$ such that $\nabla_{E}^{\prime
}(rm)=E(r)m+r\nabla_{E}^{\prime}(m)$ for all $r\in R$ and $m\in M.$ Since
$\nabla_{E}-\nabla_{E}^{\prime}\in\enm_{R}(M)_{0}=R_{0}=k$, it follows that
$\nabla_{E}=E+P+\lambda I$ for some $\lambda\in k$.

We claim that the curvature $K_{\nabla}=0$. Since $K_{\nabla}\in
\hmm_{R}(\wedge^{2}\der_{k}(R),R)$, it follows from the calculations preceding
Theorem \ref{qh_case} that it is enough to show that $K_{\nabla}(E\wedge
D_{1})=0.$ We also have $[E,D_{1}]=\frac{1}{d}(d-d_{1}-d_{2})D_{1}.$ Write
$\nabla_{D_{1}}=D_{1}+Q$, where $Q=(q_{ij})\in\enm_{R}(L_{0})$ and
$\deg(q_{ij})=\deg(e_{j})-\deg(e_{i})+(d-d_{1}-d_{2})$. Then:
\begin{align*}
K_{\nabla}(E\wedge D_{1}) &  =\nabla_{E}\nabla_{D_{1}}-\nabla_{D_{1}}%
\nabla_{E}-\nabla_{\lbrack E,D_{1}]}\\
&  =(E+P+\lambda I)(D_{1}+Q)-(D_{1}+Q)(E+P+\lambda I)\\
&  -\frac{1}{d}(d-d_{1}-d_{2})(D_{1}+Q)\\
&  =E(Q)-D_{1}(P+\lambda I)+[P+\lambda I,Q]-\frac{1}{d}(d-d_{1}-d_{2})Q\\
&  =E(Q)+[P,Q]-\frac{1}{d}(d-d_{1}-d_{2})Q
\end{align*}
A direct computation gives $[P,Q]=(-\frac{1}{d})(\deg(e_{j})-\deg
(e_{i}))(q_{ij})$, and we clearly have $E(Q)=\frac{1}{d}(d-d_{1}-d_{2}%
)Q+\frac{1}{d}(\deg(e_{j})-\deg(e_{i}))(q_{ij})$.
\end{proof}

\begin{example}
Let $R=k[x,y,z]/(x^{3}+y^{3}+z^{3}).$ The module $M$ with presentation matrix%
\[
\left(
\begin{array}
[c]{cc}%
x & -y^{2}+yz-z^{2}\\
y+z & x^{2}%
\end{array}
\right)
\]
is a maximal Cohen-Macaulay of rank one, see Laza et al \cite{LazPfiPop:02}.
The derivation module $\der_{k}(R)$ is generated by the four derivations%
\begin{align*}
E &  =x\frac{\partial}{\partial x}+y\frac{\partial}{\partial y}+z\frac
{\partial}{\partial z} & D_{1} &  =3y^{2}\frac{\partial}{\partial x}%
-3x^{2}\frac{\partial}{\partial y}\\
D_{2} &  =3z^{2}\frac{\partial}{\partial x}-3x^{2}\frac{\partial}{\partial z}
& D_{3} &  =3z^{2}\frac{\partial}{\partial y}-3y^{2}\frac{\partial}{\partial
z}%
\end{align*}
where $E$ is the Euler derivation and $D_{1},D_{2}$ and $D_{3}$ are the Kozul
derivations. Using our Singular \cite{gps05} library \textsc{Connections.lib}
\cite{EriGus:06b}, we find that a connection is represented by%
\begin{align*}
\nabla_{E} &  =E+\left(
\begin{array}
[c]{cc}%
\frac{2}{3} & 0\\
0 & \frac{2}{3}%
\end{array}
\right)   & \nabla_{D_{1}} &  =D_{1}+\left(
\begin{array}
[c]{cc}%
0 & 2x\\
-2y+z & 0
\end{array}
\right)  \\
\nabla_{D_{2}} &  =D_{2}+\left(
\begin{array}
[c]{cc}%
0 & 2x\\
y-2z & 0
\end{array}
\right)   & \nabla_{D_{3}} &  =D_{3}+\left(
\begin{array}
[c]{cc}%
-2y+2z & 0\\
0 & y-z
\end{array}
\right)  .
\end{align*}
Again using \cite{EriGus:06b}, we check that this is an integrable connection.
Further one finds that the connection represented by%
\begin{align*}
\nabla_{E}^{\prime} &  =E+\left(
\begin{array}
[c]{cc}%
\frac{2}{3} & 0\\
0 & \frac{2}{3}%
\end{array}
\right)   & \nabla_{D_{1}}^{\prime} &  =D_{1}+\left(
\begin{array}
[c]{cc}%
xz & 2x\\
-2y+z & xz
\end{array}
\right)  \\
\nabla_{D_{2}}^{\prime} &  =D_{2}+\left(
\begin{array}
[c]{cc}%
-xy & 2x\\
x^{2}+y-2z & xz
\end{array}
\right)   & \nabla_{D_{3}}^{\prime} &  =D_{3}+\left(
\begin{array}
[c]{cc}%
x^{2}-2y+2z & 0\\
0 & x^{2}+y-z
\end{array}
\right)
\end{align*}
is not integrable.

The integrability class $\operatorname*{ic}(M)=0$ in $\operatorname{H}%
_{\operatorname{Rin}}^{2}(\der_{k}(R),R)$ which means that $\nabla^{\prime}$
becomes integrable after removing terms of degree different from zero. In
fact, we see that this gives $\nabla.$

We also find that $\operatorname{H}_{\operatorname{Rin}}^{1}(\der_{k}%
(R),R)\cong k,$ which means that there is a one parameter family of integrable
connections on $M.$
\end{example}

\bibliographystyle{amsplain}
\bibliography{alggeo}

\def\cprime{$'$}
\providecommand{\bysame}{\leavevmode\hbox to3em{\hrulefill}\thinspace}
\providecommand{\MR}{\relax\ifhmode\unskip\space\fi MR }
\providecommand{\MRhref}[2]{%
  \href{http://www.ams.org/mathscinet-getitem?mr=#1}{#2}
}
\providecommand{\href}[2]{#2}
\begin{thebibliography}{10}

\bibitem{Beh:89}
Kurt Behnke, \emph{On {A}uslander modules of normal surface singularities},
  Manuscripta Math. \textbf{66} (1989), no.~2, 205--223. \MR{MR1027308
  (90k:14031)}

\bibitem{BruHer:93}
Winfried Bruns and J{\"u}rgen Herzog, \emph{Cohen-{M}acaulay rings}, Cambridge
  Studies in Advanced Mathematics, vol.~39, Cambridge University Press,
  Cambridge, 1993. \MR{MR1251956 (95h:13020)}

\bibitem{CasLadPir:05}
J.~M. Casas, M.~Ladra, and T.~Pirashvili, \emph{Triple cohomology of
  {L}ie-{R}inehart algebras and the canonical class of associative algebras},
  J. Algebra \textbf{291} (2005), no.~1, 144--163. \MR{MR2158515 (2006f:17020)}

\bibitem{eri:03}
Eivind Eriksen, \emph{Differential operators on monomial curves}, J. Algebra
  \textbf{264} (2003), no.~1, 186--198. \MR{MR1980691 (2004i:16035)}

\bibitem{EriGus:06b}
Eivind Eriksen and Trond~St{\o}len Gustavsen, \emph{{\sc Connections.lib}}, A
  {\sc singular} library to compute obstructions for existence of connections
  on modules, 2006, Available at {\tt
  http://home.hio.no/\verb1~1eeriksen/connections.html}.

\bibitem{EriGus:08}
Eivind Eriksen and Trond~St{\o}len Gustavsen, \emph{{Connections on modules
  over singularities of finite CM representation type.}}, J. Pure Appl. Algebra
  \textbf{212} (2008), no.~7, 1561--1574.

\bibitem{gps05}
G.-M. Greuel, G.~Pfister, and H.~Sch\"onemann, \emph{{\sc Singular} 3.0}, {A
  Computer Algebra System for Polynomial Computations}, Centre for Computer
  Algebra, University of Kaiserslautern, 2005, {\tt
  http://www.singular.uni-kl.de}.

\bibitem{Gro:66a}
A.~Grothendieck, \emph{On the de {R}ham cohomology of algebraic varieties},
  Inst. Hautes \'Etudes Sci. Publ. Math. (1966), no.~29, 95--103. \MR{MR0199194
  (33 \#7343)}

\bibitem{Har:80}
Robin Hartshorne, \emph{Stable reflexive sheaves}, Math. Ann. \textbf{254}
  (1980), no.~2, 121--176. \MR{MR597077 (82b:14011)}

\bibitem{HuaLukYau:03}
Xiaojun Huang, Hing~Sun Luk, and Stephen S.-T. Yau, \emph{Punctured local
  holomorphic de {R}ham cohomology}, J. Math. Soc. Japan \textbf{55} (2003),
  no.~3, 633--640. \MR{MR1978213 (2004f:32037)}

\bibitem{Hue:99}
Johannes Huebschmann, \emph{Duality for {L}ie-{R}inehart algebras and the
  modular class}, J. Reine Angew. Math. \textbf{510} (1999), 103--159.
  \MR{MR1696093 (2000f:53109)}

\bibitem{kah:88}
Constantin P.~M. Kahn, \emph{Reflexive {M}oduln auf einfach-elliptischen
  {F}l\"achensingularit\"aten}, Bonner Mathematische Schriften [Bonn
  Mathematical Publications], 188, Universit\"at Bonn Mathematisches Institut,
  Bonn, 1988, Dissertation, Rheinische Friedrich-Wilhelms-Universit\"at, Bonn,
  1988. \MR{MR930666 (90h:14045)}

\bibitem{LazPfiPop:02}
Radu Laza, Gerhard Pfister, and Dorin Popescu, \emph{Maximal {C}ohen-{M}acaulay
  modules over the cone of an elliptic curve}, J. Algebra \textbf{253} (2002),
  no.~2, 209--236. \MR{MR1929188 (2003k:13012)}

\bibitem{Mum:61}
David Mumford, \emph{The topology of normal singularities of an algebraic
  surface and a criterion for simplicity}, Inst. Hautes \'Etudes Sci. Publ.
  Math. (1961), no.~9, 5--22. \MR{27 \#3643}

\bibitem{Rin:63}
George~S. Rinehart, \emph{Differential forms on general commutative algebras},
  Trans. Amer. Math. Soc. \textbf{108} (1963), 195--222. \MR{MR0154906 (27
  \#4850)}

\bibitem{Wei:94}
Charles~A. Weibel, \emph{An introduction to homological algebra}, Cambridge
  Studies in Advanced Mathematics, vol.~38, Cambridge University Press,
  Cambridge, 1994. \MR{MR1269324 (95f:18001)}

\bibitem{YosKaw:88}
Yuji Yoshino and Takuji Kawamoto, \emph{The fundamental module of a normal
  local domain of dimension {$2$}}, Trans. Amer. Math. Soc. \textbf{309}
  (1988), no.~1, 425--431. \MR{MR957079 (89h:13033)}

\end{thebibliography}

\end{document}